\documentclass[
	english
]{scrartcl}

\usepackage[T1]{fontenc}
\usepackage[utf8]{inputenc}

\usepackage{amsmath,amsfonts,amsthm,amssymb}
\usepackage{xcolor}
\usepackage[colorlinks,citecolor=blue]{hyperref}

\usepackage{preamble}
\usepackage{marginnote}

\usepackage{enumitem}
\setenumerate{label=(\roman*)}

\usepackage[capitalise,nameinlink]{cleveref}

\usepackage{dsfont}

\usepackage{soul,cancel}

\usepackage{graphicx}

\usepackage{mathtools} 

\usepackage{lmodern}

\newif\ifbiber
\bibertrue

\ifbiber
\usepackage[
	style=authoryear-comp,
	sorting=nyt,
	url=false, 
	doi=true,                
	eprint=true,
	uniquename=allinit,        
	maxbibnames=99,          
	maxcitenames=2,
	uniquelist=minyear,
	dashed=false,            
	sortcites=false,          
	backend=biber,           
	safeinputenc,            
]{biblatex}

\DeclareCiteCommand{\cite}{%
	\ifbibmacroundef{cite:init}{}{\usebibmacro{cite:init}}\usebibmacro{prenote}%
}{%
	\usebibmacro{citeindex}%
	\printtext[bibhyperref]{\usebibmacro{cite}}%
}{%
	\ifbibmacroundef{cite:init}{\multicitedelim}{}%
}{%
	\usebibmacro{postnote}%
}%
\DeclareCiteCommand{\parencite}[\mkbibbrackets]{%
	\ifbibmacroundef{cite:init}{}{\usebibmacro{cite:init}}\usebibmacro{prenote}%
}{%
	\usebibmacro{citeindex}%
	\printtext[bibhyperref]{\usebibmacro{cite}}%
}{%
	\ifbibmacroundef{cite:init}{\multicitedelim}{}%
}{%
	\usebibmacro{postnote}%
}%

\let\cite\parencite

\else

\usepackage{doi,natbib}

\fi

\newcommand\norm[1]{\lVert#1\rVert}
\newcommand\abs[1]{\lvert#1\rvert}

\newcommand\dual[2]{\langle #1, #2\rangle}

\newcommand\scalarprod[2]{( #1, #2)}

\newcommand\R{\mathbb{R}}
\newcommand\EE{\mathcal{E}}

\renewcommand\d{\mathrm{d}}

\newcommand\capa{\operatorname{cap}}
\newcommand\supp{\operatorname{supp}}

\newcommand{\weakly}{\rightharpoonup}

\newcommand{\anni}{^\perp}
\newcommand{\polar}{^\circ}
\newcommand{\dualspace}{^\star}
\newcommand{\bidualspace}{^{\star\star}}

\newcommand\BB{\mathcal{B}}

\newcommand\LL{\mathcal{L}}
\newcommand\TT{\mathcal{T}}
\newcommand\RR{\mathcal{R}}
\newcommand\KK{\mathcal{K}}
\newcommand\SSS{\mathcal{S}}
\newcommand\MM{\mathcal{M}}
\newcommand\NN{\mathcal{N}}

\DeclareMathAlphabet{\mathpzc}{OT1}{pzc}{m}{it}

\newcommand\sign{\operatorname{sign}}
\newcommand\cl{\operatorname{cl}}
\newcommand\tr{\operatorname{tr}}

\let\subseteq\subset

\newcommand\orcid[1]{%
	\hspace{.25em}%
	\href{http://orcid.org/#1}{%
		\protect\includegraphics[height=1em]{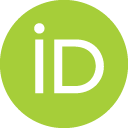}%
	}%
	\hspace{.25em}%
}

\newtheorem{theorem}{Theorem}[section]
\newtheorem{lemma}[theorem]{Lemma}
\newtheorem{proposition}[theorem]{Proposition}

\newtheorem{corollary}[theorem]{Corollary}

\newtheorem{definition}[theorem]{Definition}

\definecolor{darkgreen}{rgb}{0,0.5,0}
\definecolor{darkred}{rgb}{0.8,0,0}

\begin{document}
\title{On the Non-Polyhedricity of Sets with Upper and Lower Bounds in Dual Spaces%
\footnote{
This research was supported by the German Research Foundation (DFG) under grant numbers ME 3281/7-1 and WA 3636/4-1
 within the priority program ``Non-smooth and Complementarity-based Distributed Parameter
Systems: Simulation and Hierarchical Optimization'' (SPP 1962).}
}

\author{%
	Constantin Christof%
	\footnote{%
		Technische Universität Dortmund,
		Faculty of Mathematics,
		LS X, 
		44227 Dortmund,
		Germany
		}%
	\and
	Gerd Wachsmuth%
	\footnote{%
		Technische Universität Chemnitz,
		Faculty of Mathematics,
		Professorship Numerical Mathematics (Partial Differential Equations),
		09107 Chemnitz,
		Germany,
		\url{http://www.tu-chemnitz.de/mathematik/part_dgl/people/wachsmuth/},
		\email{gerd.wachsmuth@mathematik.tu-chemnitz.de}%
	}
	\orcid{0000-0002-3098-1503}%
}
\publishers{}
\maketitle

\begin{abstract}
We demonstrate that the set $L^\infty(X, [-1,1])$
 of all measurable functions over a Borel measure space $(X, \BB, \mu )$ with values in the unit interval is typically non-polyhedric when interpreted as a subset of a dual space.
Our findings contrast the classical result that subsets of Dirichlet spaces with pointwise upper and lower bounds are polyhedric.
In particular, additional structural assumptions are unavoidable when the concept of polyhedricity is used to study the differentiability properties of solution maps to variational inequalities of the second kind in, e.g., the spaces $H^{1/2}(\partial \Omega)$ or $H_0^1(\Omega)$.
\end{abstract}


 \begin{msc}
	\mscLink{35B30},
	\mscLink{47J22},
	\mscLink{49K40}
\end{msc}

\section{Introduction} 
As it is well known, the concept of polyhedricity is of major importance for the study of optimization problems and variational inequalities. It corresponds to a notion of ``uncurvedness''  in Banach spaces and, as such, provides a sufficient criterion for the directional differentiability of metric projections and the validity of no-gap second-order optimality conditions. We refer to, e.g., \cite{Haraux1977,Mignot1976, BonnansShapiro2000,Wachsmuth2016:2,ChristofWachsmuth2017:1} for details on these topics.
The result that is most commonly used to check the condition of polyhedricity is Mignot's classical theorem on the polyhedricity of sets with upper and lower bounds in Dirichlet spaces, cf.\ \cite[Théorème~3.2]{Mignot1976} and also the more recent contribution \cite{Wachsmuth2016:2}.
This theorem yields that, if $V \subseteq L^2(X, \mu)$ is a Dirichlet space over some Borel measure space $(X, \BB, \mu )$, then a set of the form
\begin{equation*}
  \left \{ v \in V \mid \varphi \leq v \leq \psi \ \mu\text{-a.e.\ in } X \right \}
\end{equation*}
is always polyhedric provided the bounds $\varphi$ and $\psi$ are measurable functions.
Typical examples of sets that fall into the setting of Mignot are, e.g.,
\begin{gather*}
 \left \{ v \in H^1(\Omega)\mid  -1 \leq v \leq 1 \ \LL^d\text{-a.e.\ in } \Omega\right \}, 
\\
 \left \{ v \in H^{1/2}(\partial \Omega)\mid  -1 \leq v \leq 1 \ \SSS^{d-1}\text{-a.e.\ in } \partial \Omega\right \},
\end{gather*}
where $\LL^d$ and $\SSS^{d-1}$ denote the Lebesgue and the surface measure, respectively, and $\Omega \subset \R^d$ is a domain whose boundary is Lipschitz. 

The aim of this paper is to prove that a result analogous to Mignot's theorem does not hold for sets with upper and lower bounds in the dual $V\dualspace$ of commonly used Dirichlet spaces $V$. To be more precise, in what follows, we demonstrate that the set 
\begin{equation}
\label{eq:BoundSetDual}
 L^\infty(X, [-1,1]) := \left \{ v \in L^\infty(X, \mu) \mid  -1 \leq v \leq 1 \ \mu\text{-a.e.\ in } X \right \}  
\end{equation}
is typically non-polyhedric in the dual $V\dualspace$ of a space $V$ with $V \hookrightarrow L^1(X, \mu)$ which is Dirichlet or continuously embedded into $C_b(X)$
and which satisfies mild additional assumptions.

Note that sets of the form \eqref{eq:BoundSetDual} arise naturally in the study of elliptic variational inequalities of the second kind, see, e.g, \cite{Sokolowski1988,SokolowskiZolesio1988,SokolowskiZolesio1992}. They appear when a variational inequality of the second kind involving an $L^1$-norm is transformed into a variational inequality of the first kind by dualization and are thus of relevance for questions of sensitivity analysis, cf.\ \cref{sec:motivation}.

Before we begin with our analysis, we give a short overview of the structure and the contents of this paper:

In \cref{sec:motivation}, we motivate in more detail why sets of the form \eqref{eq:BoundSetDual} are important for the sensitivity analysis of elliptic variational inequalities of the second kind. The main result of this section, \cref{thm:Differentiablity}, demonstrates that the differentiability properties of the solution operator to a variational inequality of the second kind involving a  proper, convex, lower semicontinuous and positively homogeneous functional $j$ (see \eqref{eq:VI_motivation} for the precise structure) depend largely on the curvature properties of the subdifferential $\partial j(0)$. This yields in particular that the polyhedricity of sets of the form \eqref{eq:BoundSetDual} (in dual space) constitutes a sufficient criterion for the differential stability of variational inequalities of the second kind involving an $L^1$-norm, cf.\ the examples at the end of \cref{sec:motivation}.

In \cref{sec:setswithbounds}, we demonstrate that sets of the form \eqref{eq:BoundSetDual} are indeed typically non-polyhedric when considered as subsets of the dual $V\dualspace$ of a space $V$ that is Dirichlet or continuously embedded into the bounded continuous functions $C_b(X)$.

Lastly, in \cref{sec:conclusion}, we summarize and interpret our findings. Here, we also address the consequences that our results have for the study of, e.g, contact problems with prescribed friction. 

\section{Motivation}
\label{sec:motivation}

To motivate our analysis, we consider an elliptic variational inequality of the second kind, i.e., a problem of the form
\begin{equation}
	\label{eq:VI_motivation}
	\text{Find }y \in V:\qquad 
	\dual{A \, y - f}{v - y}
	+
	j(v) - j(y)
	\ge
	0
	\qquad\forall v \in V.
\end{equation}
Our assumptions on the quantities in \eqref{eq:VI_motivation} are as follows:
\begin{enumerate}
\item $V$ is a real Hilbert space with dual $V\dualspace$ and dual pairing $\left \langle \cdot, \cdot \right \rangle$,
\item $A : V \to V\dualspace$ is a linear, bounded and coercive mapping, i.e., 
\begin{equation*}
	\exists c,C > 0 : \
	c \, \norm{v}_V^2 \le \dual{A \, v}{v} \le C \, \norm{v}_V^2 \quad \forall v \in V,
\end{equation*}
\item $f \in V\dualspace$ is a given datum,
\item $j : V \to (-\infty, \infty]$
is a proper, convex and lower semicontinuous functional that is positively homogeneous of degree one, i.e., $j(\alpha v ) = \alpha j(v)$ for all $v \in V$, $\alpha > 0$.
\end{enumerate}
Note that we could also work with a Banach space $V$ in the above. This, however, would just pretend more generality. It is well known that the existence of a  linear, bounded and coercive mapping implies that  a Banach space is isomorphic to a Hilbert space. Thus, we restrict our attention w.l.o.g.\ to the Hilbert space setting.

From standard results in convex analysis and from the theory of  elliptic variational inequalities, we obtain the following result concerning problem \eqref{eq:VI_motivation}.

\begin{proposition}
	\label{proposition:SolvabilityDuality}
	Problem \eqref{eq:VI_motivation} admits a unique solution $y \in V$ for all $f \in V\dualspace$. This solution satisfies $y = A^{-1}(f-q)$, where $q \in V\dualspace$ is the unique solution to the  elliptic variational inequality 
	\begin{equation}
	\label{eq:VI_motivation_dual}
		\text{Find }q \in \partial j(0):\qquad 
		\dual{p-q}{A^{-1}(q-f)} \ge 0
		\qquad
		\forall p \in \partial j(0).
	\end{equation}
\end{proposition}
\begin{proof}
	The operator $A^{-1} : V\dualspace \to V$ is trivially linear, bounded and coercive, and the subdifferential $\partial j(0)$ is non-empty, closed and convex.
	Thus, \eqref{eq:VI_motivation_dual} is a $V\dualspace$-elliptic variational inequality of the first kind.
	The classical result \cite[Theorem~II.2.1]{KinderlehrerStampacchia1980} implies that \eqref{eq:VI_motivation_dual} possesses a unique solution.

	Hence, it suffices to show that
	$y$ solves \eqref{eq:VI_motivation}
	if and only if $q := f - A \, y$ solves \eqref{eq:VI_motivation_dual}.
	Now we have the equivalencies
	\begin{align}
		\label{eq:equivalencies_0}
		\text{$y$ solves \eqref{eq:VI_motivation}} \quad\Leftrightarrow\quad& 
		q \in \partial j(y)
		\\
		\label{eq:equivalencies_1}
		\Leftrightarrow \quad &
		q \in \partial j(0) \;\text{and}\; \dual{q}{y} = j(y)
		\\
		\label{eq:equivalencies_2}
		\Leftrightarrow \quad &
		q \in \partial j(0) \;\text{and}\; \dual{q - p}{y} \ge 0 \qquad\forall p \in \partial j(0)
		\\
		\notag
		\Leftrightarrow \quad &
		\text{$q$ solves \eqref{eq:VI_motivation_dual}}
		.
	\end{align}
	The equivalence between \eqref{eq:equivalencies_0} and \eqref{eq:equivalencies_1}
	follows from the positive homogeneity of $j$.
	Further, ``\eqref{eq:equivalencies_2} $\Rightarrow$ \eqref{eq:equivalencies_1}''
	can be seen as follows.
	Using again the positive homogeneity of $j$, we have
	\begin{equation*}
		j(y) = \sup_{p \in \partial j(0)} \dual{p}{y},
	\end{equation*}
	see \cite[Proposition~16.18]{BauschkeCombettes2011}.
	This means that \eqref{eq:equivalencies_2} implies
	$j(y) = \dual{q}{y}$ and \eqref{eq:equivalencies_1} follows.
	The remaining implications above are obvious and this finishes the proof.
\end{proof}
Note that the above proof shows that the mapping
$f \mapsto y$ is Lipschitz continuous as a function from $V\dualspace$ to $V$.

The identity $y = A^{-1}(f-q)$ in \cref{proposition:SolvabilityDuality} implies that, if we are interested in the differentiability properties of the solution map $S : V\dualspace \to V$, $f \mapsto y$, associated with \eqref{eq:VI_motivation}, then we may equivalently study the differentiability properties of the solution operator $T : V\dualspace \to V\dualspace$, $f \mapsto q$, associated with \eqref{eq:VI_motivation_dual}.  Note that the latter is seemingly easier since \eqref{eq:VI_motivation_dual} is a variational inequality of the first kind and since differentiability results for variational inequalities of the first kind are well known, cf.\ \cite{Mignot1976,Haraux1977}. These classical results, however, are only applicable if the admissible set of the problem at hand satisfies the condition of polyhedricity, cf.\ \cite[Definition 3.51]{BonnansShapiro2000}.

\begin{definition}
\label{def:polyhedricity}
Suppose that $K$ is a closed, convex and non-empty subset of a Banach space $W$. Then $K$ is said to be polyhedric at a point $w \in K$ w.r.t.\ $\eta \in \NN_K(w)$ if
\begin{equation*}
\TT_K(w) \cap \eta\anni = \overline{ \RR_K(w) \cap \eta\anni }.
\end{equation*}
Here,
	\begin{equation*}
		\RR_K(w) := \R^+ \left ( K - w\right ),\quad 
		\TT_K(w)
		:=
		\overline{ \RR_K(w)}
		,\quad\text{and}\quad
		\NN_K (w) 
		:= \TT_K(w)^\circ
	\end{equation*}
	are the radial, the tangent and the normal cone to $K$ at $w$, respectively.
	By $\eta\anni$ we denote the kernel of the functional $\eta \in W\dualspace$ and $C\polar$ denotes the polar cone of a given cone $C \subset W$.
The set $K$ is called polyhedric  if it is polyhedric at every point $w \in K$ w.r.t.\ all  $\eta \in \NN_K(w)$.
\end{definition}

Under the assumption of polyhedricity, it is straightforward to obtain the following extension of \cite[Theorem 4.2]{Hintermueller2017},
see also \cite{Sokolowski1988,SokolowskiZolesio1988,SokolowskiZolesio1992}.

\begin{theorem}
	\label{thm:Differentiablity}
	Let $f \in V\dualspace$ be given. We denote by $y := S(f)$ the solution to \eqref{eq:VI_motivation} and by $q := T(f)$ the solution to \eqref{eq:VI_motivation_dual}. Suppose that the set $\partial j(0) \subset V\dualspace$
	is polyhedric at $q$ w.r.t.\ $y \in V \cong V\bidualspace$. Then, the solution operator $S : V\dualspace \to V$ associated with \eqref{eq:VI_motivation} is Hadamard directionally differentiable in $f$ in all directions $g \in V\dualspace$ and
	the directional derivative $\delta := S'(f; g)$ in $f$ in a direction $g$ is uniquely characterized by the variational inequality 
	\begin{equation}
	\label{eq:derivative_VI}
	\text{Find }
	\delta \in 
	\bigh(){ \TT_{\partial j(0)}(q) \cap y\anni }\polar
	:
	\qquad
	\dual{A \, \delta - g}{v - \delta}
	\ge
	0
	\qquad \forall v \in \bigh(){ \TT_{\partial j(0)}(q) \cap y\anni }\polar.
	\end{equation}
\end{theorem}

\begin{proof}
Since $\partial j(0)$ is polyhedric at $q$ w.r.t.\ $y = A^{-1}(f - q)$, we may employ the classical results of Mignot and Haraux, see \cite[Théorème 2.1]{Mignot1976} and \cite[Theorem 2]{Haraux1977}, to deduce that the map $T : V\dualspace \to V\dualspace$ is directionally differentiable in $f$ in all directions $g \in V\dualspace$, and that the directional derivative $\eta := T'(f;g)$ in $f$ in a direction $g$ is uniquely characterized by the variational inequality 
	\begin{equation}
	\label{eq:derivative_VI_dual}
		\text{Find }\eta \in \TT_{\partial j(0)}(q) \cap y\anni:\qquad 
		\dual{z - \eta }{A^{-1}( \eta - g)} \ge 0
		\qquad
		\forall z \in \TT_{\partial j(0)}(q) \cap y\anni.
	\end{equation}
From the relation between the solutions to \eqref{eq:VI_motivation} and \eqref{eq:VI_motivation_dual}, it is evident that the map $S : V\dualspace \to V$ is directionally differentiable in $f$ in all directions $g \in V\dualspace$ with $S'(f; g) = A^{-1}( g - T'(f; g))$.
Next, we check \eqref{eq:derivative_VI}. To this end, fix a direction $g \in V\dualspace$ and write $\delta := S'(f;g)$, $\eta := T'(f; g)$. Then, \eqref{eq:derivative_VI_dual} and the identity $\delta = A^{-1}(g - \eta)$ imply
\begin{equation*}
	\dual{z - \eta }{\delta} \le 0
	\qquad
	\forall z \in \TT_{\partial j(0)}(q) \cap y\anni.
\end{equation*}
Hence,
\begin{equation*}
	\delta \in
	\bigh(){ \TT_{\partial j(0)}(q) \cap y\anni }\polar
	\quad\text{and}\quad
	\dual{\eta }{\delta} = 0.
\end{equation*}
For all
$v \in \bigh(){ \TT_{\partial j(0)}(q) \cap y\anni }\polar$,
we obtain on the other hand
\begin{equation*}
	\dual{g - A \, \delta}{v - \delta}
	=
	\dual{\eta}{v - \delta}
	=
	\dual{\eta}{v}
	\le
	0.
\end{equation*}
Combining the above yields \eqref{eq:derivative_VI}.
Finally, the Hadamard directional differentiability of $S$ follows from \cite[Proposition~2.49]{BonnansShapiro2000} and the global Lipschitz continuity of $S$.
\end{proof}

\cref{thm:Differentiablity} illustrates that it makes sense to study the polyhedricity of the set $\partial j(0)$.
In applications, the functional $j$ often takes the form 
$j(v) = \norm{G v}_{L^1(X, \mu)}$,
where $G : V \to L^1(X, \mu)$ is a bounded linear mapping
and $(X, \Sigma, \mu)$ is a $\sigma$-finite measure space, cf., e.g., \cite{Sokolowski1988}.
For such a function $j$, the chain rule for subdifferentials, see \cite[Proposition 5.7]{Ekeland1976}, implies
\begin{equation}
\label{eq:L1subdifferential}
	\partial j(0)
	=
	G\dualspace \, \partial \norm{\cdot}_{L^1(X, \mu)}(0),
\end{equation}
where
\begin{equation*}
	\partial \norm{\cdot}_{L^1(X, \mu)}(0)
	=
	\{ v \in L^\infty(X, \mu) \mid -1 \le v \le 1 \ \mu\text{-a.e.\ in }X \} =:M.
\end{equation*}
In particular, the polyhedricity condition in \cref{thm:Differentiablity} becomes
\begin{equation*}
\TT_{G\dualspace M}(q) \cap y\anni = \overline{\RR_{G\dualspace M}(q) \cap y\anni},
\end{equation*}
so that we indeed end up with the same assumption as in \cite[Section 4]{Hintermueller2017}. To get an intuition for the structure of the set in \eqref{eq:L1subdifferential}, let us consider two  tangible examples. 

First, let $\Omega \subset \R^d$ be a bounded, open set. Set $V := H_0^1(\Omega)$, $X := \Omega$ and $\mu := \LL^d$, and suppose that 
 $G$ is the injection of $H_0^1(\Omega)$ into $L^1(\Omega)$.
Then, 
\begin{equation*}
	\partial j(0)
	=
	\{ v \in L^\infty(\Omega) \mid -1 \le v \le 1 \ \LL^d\text{-a.e.\ in } \Omega \}
	\subset
	H^{-1}(\Omega)
\end{equation*}
and we arrive at a set that has precisely the form \eqref{eq:BoundSetDual}.

For the second example, assume that $\Omega \subset \R^d$ is a bounded domain  
with a Lipschitz boundary and set $V := H^1(\Omega)$, $X := \partial\Omega$ and $\mu := \SSS^{d-1}$, where $\SSS^{d-1}$ is the boundary measure on $\partial \Omega$. Consider the function $G: H^1(\Omega) \to L^1(\partial \Omega)$, $v \mapsto \tr(v)$, where $\tr$ denotes the trace operator $\tr: H^1(\Omega) \to H^{1/2}(\partial \Omega)$.
Then,
\begin{equation*}
	\partial j(0)
	=
	\{ v \in L^\infty(\partial\Omega) \mid -1 \le v \le 1 \ \SSS^{d-1}\text{-a.e.\ on } \partial\Omega \}
	\subset
	H^{1}(\Omega)\dualspace.
\end{equation*}
Via the trace operator, we may identify the set $H^{-1/2}(\partial \Omega)$ with a closed subspace of $H^{1}(\Omega)\dualspace$ (namely, $\tr\dualspace H^{-1/2}(\partial \Omega)$). Further, we have the following result.

\begin{lemma}
	\label{lem:polyhedric_in_subspace}
	Let $W$ be a Banach space and let $U \subset W$ be a closed subspace of $W$.
	Then, a closed, convex and non-empty set $K \subset U$
	is polyhedric as a subset of the Banach space $W$
	if and only if $K$ is polyhedric as a subset of the Banach space $U$.
\end{lemma}
\begin{proof}
	Since $U$ and $W$ are equipped with the same norm,
	the tangent cones to $K$ in $U$ and $W$ coincide.
	According to \cite[Lemma~4.1]{Wachsmuth2016:2}, to establish the claim of the lemma, it suffices to show that the statements
	\begin{equation}
		\label{eq:poly_in_U}
		\TT_K(w) \cap \lambda\anni
		=
		\overline{ \RR_K(w) \cap \lambda\anni}
		\qquad\forall \lambda \in U\dualspace
	\end{equation}
	and
	\begin{equation}
		\label{eq:poly_in_W}
		\TT_K(w) \cap \eta\anni
		=
		\overline{\RR_K(w) \cap \eta\anni}
		\qquad\forall \eta\in W\dualspace
	\end{equation}
	are equivalent for all $w \in K$.
	We first prove ``\eqref{eq:poly_in_U} $\Rightarrow$ \eqref{eq:poly_in_W}''. Suppose that an $\eta \in W\dualspace$ is given. Then,  $\lambda := \eta|_U$ is in  $U\dualspace$, and it holds 
	$\TT_K(w), \RR_K(w) \subset U$. Consequently, 
	\begin{equation*}
		\TT_K(w) \cap \eta\anni
		=
		\TT_K(w) \cap \lambda\anni
		=
		\overline{
		\RR_K(w) \cap \lambda\anni}
		=
		\overline{
		\RR_K(w) \cap \eta\anni}.
	\end{equation*}
	This proves  \eqref{eq:poly_in_W}. To obtain ``\eqref{eq:poly_in_W} $\Rightarrow$ \eqref{eq:poly_in_U}'', we can proceed along exactly the same lines by
	using the theorem of Hahn-Banach to extend $\lambda \in U\dualspace$ to a functional $\eta \in W\dualspace$.
\end{proof}

The above lemma shows that it does not make any difference whether we discuss the polyhedricity of the set $L^\infty(\partial\Omega, [-1,1])$
in the space $H^1(\Omega)\dualspace$ or in its closed subspace $\tr\dualspace H^{-1/2}(\partial \Omega) \subset H^1(\Omega)\dualspace$. Since the spaces 
\begin{equation*}
(\tr\dualspace H^{-1/2}(\partial \Omega), \|\cdot \|_{H^1(\Omega)\dualspace})\quad \text{and}\quad (H^{-1/2}(\partial \Omega), \|\cdot \|_{H^{-1/2}(\partial \Omega)})
\end{equation*}
 are isomorphic, cf.\ the inverse trace theorem, we may simplify the situation further and confine ourselves to studying the set  $L^\infty(\partial\Omega, [-1,1])$ as a subset of $H^{-1/2}(\partial \Omega)$. This is again the setting in \eqref{eq:BoundSetDual}. Note that the situation that we have considered here is exactly that studied in \cite[Section 4.5]{SokolowskiZolesio1992}. 

The above examples demonstrate that, if we want to  apply \cref{thm:Differentiablity} to a variational inequality of the second kind involving an $L^1(X, \mu)$-norm, then we typically have to check the condition of polyhedricity for a set of the form  \eqref{eq:BoundSetDual} in the dual of the underlying space $V$. Compare with  \cite{Sokolowski1988,SokolowskiZolesio1988,SokolowskiZolesio1992,Hintermueller2017,DelosReyes2016} in this context. The problem is that a set of the form \eqref{eq:BoundSetDual} is typically not polyhedric when considered as a subset of a dual space. This unfortunate and quite counterintuitive fact is proved in the following section.

\section{Sets with bounds in dual spaces}
\label{sec:setswithbounds}

\noindent
Let $(X, \Sigma, \mu)$ be a $\sigma$-finite measure space and suppose that $V$ is a Banach space which embeds densely into $L^1(X, \mu)$. We are interested in the non-polyhedricity of the set
\begin{equation*}
	M
	:=
	\{v \in L^\infty(X, \mu) : -1 \le v \le 1 \ \mu\text{-a.e.\ in } X \}
\end{equation*}
as a subset of the dual space $V\dualspace$. Note that $M$ is trivially convex, non-empty and closed in $V\dualspace$ (since we may again identify $M$ with a subdifferential, cf.\ \cref{sec:motivation}).

In order to check that $M$ is not polyhedric,
it suffices to find a
$q \in M$ and a  $y \in V$ such that $y \in \NN_M(q) \subset V\bidualspace$ and 
\begin{equation}
	\label{eq:counterexample}
	\exists
	\nu \in \TT_M(q) \cap y\anni
	\ \text{with} \ 
	\nu \not\in \overline{\RR_M(q) \cap y\anni }
	.
\end{equation}
To this end,
we suppose that 
a tuple $(q,y) \in M \times V$ is given such that $y \in \NN_M(q)$ and such that 
there exists an $O \in \Sigma$
with  $\mu(O \cap \{y = 0\})=0$ and $\mu(O) > 0$.
From the definition of $M$ and $y \in \NN_M(q)$, we obtain that it holds  $q = \sign(y)$ $\mu$-a.e.\ in $O$. 
Now, for any $\eta \in \RR_M(q)$ with $\eta \, \chi_O \neq 0$ in $L^\infty(X, \mu)$ (where $\chi_O : X \to \{0,1\}$ is the characteristic function of the set $O$)
we have
\begin{equation*}
	\dual{\eta}{y}
	=
	\int_X \eta\, y \, \d \mu
	=
	-\int_X \abs{\eta} \, \abs{y} \, \d \mu
	\le
	-\int_O \abs{\eta} \, \abs{y} \, \d \mu
	< 0.
\end{equation*}
This shows that, to prove the non-polyhedricity of the set $M$, it is enough to construct a 
$\nu \in \TT_M(q) \cap y\anni$
with the property
that for each $\{\nu_n\} \subset L^\infty(X, \mu)$
with $\nu_n \to \nu$ in $V\dualspace$
we have $\nu_n \, \chi_O \ne 0$ for $n$ large enough.
Indeed, for such a $\nu$,
any sequence
$\{\nu_n\} \subset \RR_M(q)$
with $\nu_n \to \nu$ in $V\dualspace$
satisfies $\dual{\nu_n}{y} < 0$ for $n$ large enough so that $\nu_n \not \in \RR_M(q) \cap y\anni$.

In the following two sections,
we demonstrate how a $\nu$ with the above properties can be constructed if $V$ is a Dirichlet space or if $V$ embeds into the space $C_b(X)$.

\subsection{Dirichlet spaces}
\label{subsec:DirichletSpaces}
In this section, we study the non-polyhedricity of the set \eqref{eq:BoundSetDual} in the dual $V\dualspace$ of a Dirichlet space $V$. 
For convenience, we begin by briefly recalling the assumptions typically made in the Dirichlet space setting, see \cite{FukushimaOshimaTakeda2011} for more details.

In what follows, we assume that $X$ is a locally compact and separable metric space.
We denote the  Borel $\sigma$-algebra of $X$ with $\BB$ and assume that 
  $\mu: \BB \to [0, \infty]$ is a   Borel measure 
such that $\mu(O) > 0$ holds for all non-empty open sets $O \subset X$
and such that
$\mu(K) < \infty$ holds for all compact sets $K \subset X$.
Note that the  measure $\mu$ is automatically regular 
since $X$ is $\sigma$-compact, cf.\ \cite[Theorem~2.18]{Rudin1987}.

Now, a subspace $V \subset L^2(X, \mu)$ is called a Dirichlet space
if there exists a symmetric, bilinear form $\EE : V \times V \to \R$
such that
\begin{itemize}
	\item
		$V$ is dense in $L^2(X, \mu)$ and $\EE(u,u) \ge 0$ for all $u \in V$,
	\item
		$V$ is a Hilbert space  when endowed with the product
		$\EE_1(\cdot, \cdot ) = \EE(\cdot, \cdot) + \scalarprod{\cdot}{\cdot}_{L^2}$,		
	\item
		for all $u \in V$ it holds
		$v := \min(1, \max(0, u)) \in V$
		with
		$\EE(v,v) \le \EE(u,u)$.
\end{itemize}
A Dirichlet space $V$ is said to be \emph{regular}
if
\begin{itemize}
	\item
		$V \cap C_c(X)$ is dense in $V$
		w.r.t.\ the norm induced by $\EE_1$,
	\item
		$V \cap C_c(X)$ is dense in $C_c(X)$
		w.r.t.\ the supremum norm.
\end{itemize}
Recall that 
\begin{align*}
	C_c(X) &:= \{ \psi \in C(X) \mid \supp(\psi) \text{ is compact}\}, \\
	C_0(X) &:= \{ \psi \in C(X) \mid \forall \varepsilon > 0 : \exists \text{ compact } K \subset X : \abs{\psi} \le \varepsilon \text{ on } X \setminus K\}.
\end{align*}
Note that $C_c(X)$ can be replaced by $C_0(X)$ in the definition of a regular Dirichlet space
since $C_0(X)$ is the completion of $C_c(X)$ in the supremum norm.

As an important example,
we mention that, for an open, bounded set $\Omega \subset \R^d$,
the Sobolev space $V = H_0^1(\Omega)$ is a regular Dirichlet space
with $\EE(u,v) = \int_\Omega \nabla u \cdot \nabla v\,\d x$.
However, $H_0^2(\Omega)$ is not a Dirichlet space
since $\min(1, \max(0, u))$ is typically not an element of $H_0^2(\Omega)$ for a given $u \in H_0^2(\Omega)$.

Let us recall some details concerning
capacity theory in Dirichlet spaces,
see \cite[Chapter~2]{FukushimaOshimaTakeda2011}.
For an arbitrary set $A \subset X$, we define
the capacity of $A$ by
\begin{equation*}
	\capa(A)
	:=
	\inf
	\bigh\{\}{
		\EE_1(u,u)
		\mid
		u \ge 1 \text{ $\mu$-a.e.\ on an open neighborhood of } A
	}.
\end{equation*}
We say that a pointwise property holds quasi-everywhere (q.e.)
if it holds up to a set of capacity zero. 
A function $u : X \to \R$ is called
quasi-continuous
if
for every $\varepsilon > 0$
there is an open set $G \subset X$
with $\capa(G) < \varepsilon$ and $u |_{X \setminus G} \in C(X \setminus G)$.
Note that every $v \in V$
possesses a quasi-continuous representative which is unique up to sets of capacity zero,
see \cite[Theorem~2.1.3]{FukushimaOshimaTakeda2011}.
In the sequel, we will always work with the quasi-continuous
representatives of the involved functions. We begin our analysis with the following result.
\begin{theorem}
	\label{thm:dirichlet}
	Let $V$ be a regular Dirichlet space which embeds densely into $L^1(X, \mu)$. Assume that a tuple $(q, y) \in M \times  V$ is given such that  $y \in \NN_M(q) \subset V\bidualspace \cong V$ 	and such that there exists an open set $O \subset X$ with $\mu(O \cap \{y = 0\}) = 0$ and $\capa(O \cap \{y = 0\}) > 0$.
	Then, $M \subset V\dualspace$ is not polyhedric in $q$ w.r.t.\ $y$. 
\end{theorem}
\begin{proof}
	To establish the claim, it suffices to construct a $\nu$ with the properties in \eqref{eq:counterexample}.
	Since $O \cap \{y = 0\} \in \BB$,
	Choquet's capacity theorem yields that there exists
	a compact set $K \subset O \cap \{y = 0\}$
	with positive capacity,
	see \cite[(2.1.6)]{FukushimaOshimaTakeda2011}.

	In what follows, our aim is to construct a non-zero measure in $V\dualspace$
	which lives on the set $K$. To this end, 
	we define 
	$\KK := \{v \in V \mid v \ge 1 \text{ q.e.\ on } K \}$ and 
	note that the element $w$ in $\KK$ with the smallest $\EE_1$-norm
	satisfies $\capa(K) = \EE_1(w,w)$, see \cite[Theorem~2.1.5]{FukushimaOshimaTakeda2011}.
	Define $\eta := \EE_1(w,\cdot) \in V\dualspace \setminus\{0\}$. Then, it holds
	\begin{equation*}
		\dual{\eta}{v - w}_V \ge 0
		\qquad\forall v \in \KK.
	\end{equation*}
	The above implies that $\eta$ is a positive Radon measure with $\supp(\eta) \subset K$, see \cite[Lemma~2.2.6]{FukushimaOshimaTakeda2011}.
	By invoking \cite[Lemma~2.2.2]{FukushimaOshimaTakeda2011},
	we find that there exists a sequence $\{\eta_n\} \in L^2(X, \mu)$
	with $\eta_n \weakly \eta$ in $V\dualspace$ and $\eta_n \ge 0$ $\mu$-a.e.
	Write
	\begin{equation*}
		\eta_n^1 := \chi_{\{y > 0\}} \, \eta_n,
		\qquad
		\eta_n^2 := \chi_{\{y < 0\}} \, \eta_n.
	\end{equation*}
	Then, for every $v \in V$, we have
	\begin{equation*}
		\dual{\eta_n^i}{v}_{V}
		=
		\scalarprod{\eta_n^i}{v}_{L^2}
		\le
		\scalarprod{\eta_n^i}{\abs{v}}_{L^2}
		\le
		\scalarprod{\eta_n}{\abs{v}}_{L^2}
		\le
		\norm{\eta_n}_{V\dualspace} \, \norm{v}_{V},\qquad i=1,2.
	\end{equation*}
	Hence, the sequences $\{\eta_n^i\}$, $i=1,2$, are bounded in $V\dualspace$, and we may assume 
	w.l.o.g.\ that $\eta_n^i \weakly \eta^i$ holds for some $\eta^i \in V\dualspace$. Since $\eta^1 + \eta^2 = \eta$,
	at least one $\eta^i$ has to be non-zero.
	We assume w.l.o.g.\ that this is the case for $\eta^1$.
	Since $\eta^1,\eta^2 \ge 0$, we infer
	$0 \le \eta^1 \le \eta$.
	In particular, $\supp(\eta^1) \subset \supp(\eta) \subset K$.

	Consider now again the sequence  $\{\eta_n^1\}$ with $\eta_n^1 \weakly \eta^1$ in $V\dualspace$. 
	Since $\eta_n^1 \in L^2(X, \mu)$, we can choose $M_n > 0$
	such that
	$\norm{\max(\eta_n^1 - M_n,0)}_{L^2} \le \frac1n$.
	The latter yields
	\begin{equation*}
		\scalarprod{\min(\eta_n^1, M_n)}{v}_{L^2}
		=
		\scalarprod{\eta_n^1}{v}_{L^2}
		-
		\scalarprod{\max(\eta_n^1 -M_n, 0)}{v}_{L^2}
		\to
		\left \langle \eta^1, v \right \rangle \qquad \forall v \in V.
	\end{equation*}
	Hence,
	$L^\infty(X, \mu) \ni \min(\eta_n^1, M_n) \weakly \eta^1$ as $n \to \infty$.
	Since  $q = 1$ holds $\mu$-a.e.\ on $\{y > 0\}$ and since $\eta_n^1 = \chi_{\{y > 0\}} \, \eta_n$, 
	the functions $-\min(\eta_n^1, M_n)$
	belong to $\RR_M(q)$.
	From Mazur's lemma, it now follows straightforwardly that $-\eta^1 \in \TT_M(q)$. Define $\nu := - \eta^1$. 
	Then, by \cite[Theorem~2.2.2]{FukushimaOshimaTakeda2011},
	we may write the dual pairing between $\nu$ and $y$ as an integral and infer
	\begin{equation*}
		\dual{\nu }{y}
		=
		\int_X y \, \d\nu
		=
		\int_K y \, \d\nu 
		=
		0,
	\end{equation*}
	where we used that $\supp(\nu) \subset K$ and $y = 0$ q.e.\ on $K$.
	This shows
	$\nu \in \TT_M(q) \cap y\anni$.

	Finally, let $\{\nu_n\} \subset L^\infty(X, \mu)$ be an arbitrary sequence with $\nu_n \to \nu$ in $V\dualspace$. To conclude the proof of the non-polyhedricity, 
	we have to show that $\nu_n \, \chi_O \ne 0$ holds for all large enough $n$, cf.\ the discussion at the beginning of this section.
	From Urysohn's lemma, we obtain that there exists a $\hat \varphi \in C_0(X)$ with $\hat \varphi = 1$ on $K$
	and $\hat \varphi = 0$ on $X \setminus O$.
	Since $V$ is a regular Dirichlet space, we may find a $\tilde\varphi \in C_0(X) \cap V$
	with $\norm{\hat \varphi - \tilde\varphi}_{L^\infty} \le 1/3$.
	Now, $\varphi := \min(1, 3 \, \max(\tilde\varphi -1/3, 0)) \in C_0(X) \cap V$
	satisfies $\varphi = 1$ on $K$ and $\varphi = 0$ on $X \setminus O$.
	This implies
	\begin{equation*}
		\int_O \nu_n \, \varphi \, \d \mu
		=
		\scalarprod{\nu_n}{\varphi}_{L^2}
		\to
		\dual{\nu}{\varphi}
		=
		\nu(X)
		\ne 0.
	\end{equation*}
	Hence, $\chi_O \, \nu_n \ne 0$ will hold for $n$ large enough and the proof is complete.
\end{proof}

As a corollary of \cref{thm:dirichlet}, we obtain the following result. 

\begin{theorem}
	\label{thm:general_nonpolyhedric}
	Let $V$ be a regular Dirichlet space which embeds densely into $L^1(X, \mu)$.
	Suppose that for every compact set $K \subset X$ there exist a function $y \in V$ and an open set $O \subset X$ with $K = O \cap \{y = 0\}$ up to a set of zero capacity.
	Assume further that there is a set $A \in \BB$
	with $\mu(A) = 0$ and $\capa(A)>0$.
	Then, $M$ is not polyhedric.
\end{theorem}
\begin{proof}
	We have to construct a tuple $(q, y) \in M \times V$
	such that $y \in  \NN_M(q) \subset V\bidualspace \cong V$ and such that $M$ is not polyhedric in $q$ w.r.t.\ $y$.
	By Choquet's theorem, there is a compact set $K \subset A$
	with non-zero capacity and measure zero.
	Let $y \in V$ and an open set $O \subset X$ be given such that $K = O \cap \{y = 0\}$ up to a set of zero capacity.
	We set $q := \sign(y) \in M$.
	Then, it trivially holds $y \in \NN_M(q)$, and we may apply \cref{thm:dirichlet}
	to obtain that $M$ is not polyhedric in $q$ w.r.t.\ $y$. This completes the proof. 
\end{proof}

We would like to point out that the assumption ``\emph{For every compact set $K \subset X$ there exist a function $y \in V$ and an open set $O \subset X$ with $K = O \cap \{y = 0\}$ up to a set of zero capacity}'' appearing in \cref{thm:general_nonpolyhedric} is not restrictive at all. This condition is, e.g., satisfied when the Lipschitz functions belong to $V$
(in this case $y$ can be constructed with the distance function) or if $X = \Omega \subset \R^d$ is an open set and $C_c^\infty(\Omega) \subset V$, see \cite[Lemma~A.1]{ChristofClasonMeyerWalther2017}. Similarly, the existence of a set with zero measure but non-zero capacity
is ensured in many Dirichlet spaces, cf.\  \cite[Theorem~5.5.1]{AdamsHedberg1996}. Some tangible examples can be found in the following corollary.

\begin{corollary}
\label{corollary:tangible1}
Let $\Omega \subset \R^d$ be a bounded domain with a Lipschitz boundary. Then, the following holds true:
\begin{enumerate}
\item The set $  \{ v \in L^\infty(\Omega) \mid  -1 \leq v \leq 1 \ \LL^d\text{-a.e.\ in } \Omega   \}  $ is non-polyhedric in the dual space $H^{-s}(\Omega)$ of the space $H_0^s(\Omega) := \cl_{\|\cdot \|_{H^s}}(C_c^\infty(\Omega))$ for all $s \in (0, 1]$.
\item The set $  \{ v \in L^\infty(\partial \Omega) \mid  -1 \leq v \leq 1 \ \SSS^{d-1}\text{-a.e.\ on } \partial \Omega   \}  $ is non-polyhedric in the dual space $H^{-1/2}(\partial \Omega)$ of the space $H^{1/2}(\partial \Omega)$.
\end{enumerate}
\end{corollary}

\begin{proof}
It is well known that the spaces $H_0^s(\Omega)$, $s \in (0, 1]$, and $H^{1/2}(\partial \Omega)$ are regular Dirichlet spaces, cf.\ \cite{MusinaNazarov2017} and also \cite[Remark 3.2(c)]{ChristofMueller2017}. The existence of sets with zero measure but non-zero capacity in these spaces 
follows, e.g., from \cite[Theorem~5.5.1]{AdamsHedberg1996} (for $H^{1/2}(\partial \Omega)$, we can use a rectification argument here). Lastly, given a compact set $K$, we can construct a 
 tuple $(y, O)$ as in \cref{thm:general_nonpolyhedric} by multiplying the distance function $\mathrm{dist}(\cdot, K)$ with an appropriate cut-off function. The claim now follows immediately from  \cref{thm:general_nonpolyhedric}. 
\end{proof}

\subsection{Spaces embedding into the continuous functions}
\label{subsec:ContinuousSpaces}
In this section, we prove a result analogous to \cref{thm:dirichlet} in the case that $V$ embeds into the space $C_b(X)$ of bounded continuous functions. We assume, for simplicity,  that $X$, $\BB$ and $\mu$ have the same properties as in  \cref{subsec:DirichletSpaces}.

\begin{theorem}
	\label{thm:special_continuous}
	Assume that $V$ is a reflexive Banach space that embeds densely into $L^1(X, \mu)$ and continuously into $C_b(X)$. Suppose that a tuple $(q, y) \in M \times V$ is given such that $y \in \NN_M(q) \subset V\bidualspace \cong V$ and such that there exists an open set $O \subset X$ with $\mu( O \cap \{y = 0\} ) =0$ and $ O \cap \{y = 0\}  \neq \emptyset$.
	Further, suppose that for every $x \in O$
	there is a function $\varphi \in V$ with
	$\varphi(x) = 1$ and $\varphi = 0$ on $\Omega \setminus O$.
	Then, $M$ is not polyhedric in $q$ w.r.t.\ $y$.
\end{theorem}

We, of course, use the continuous representatives of the functions $y$ and $\varphi$ in the above to sensibly define the set $ O \cap \{y = 0\}$ and the pointwise evaluation of $\varphi$.

\begin{proof}
	Fix a point $x \in O \cap \{y = 0\}$ and define 
	\begin{equation*}
	\eta_n^1 :=  \mu(B_{1/n}(x))^{-1}\, \chi_{\{y > 0\} \cap B_{1/n}(x)},\qquad \eta_n^2 := \mu( B_{1/n}(x))^{-1} \, \chi_{\{y < 0\} \cap B_{1/n}(x)},
	\end{equation*}
	where $B_{1/n}(x)$ is the ball around $x$ with radius $1/n$.
	Then, $\eta^1_n + \eta^2_n$ converge weakly to the Dirac measure at $x$ in $V\dualspace$
	and (along a subsequence), one of the sequences $\{\eta_n^1\}$, $\{\eta_n^2\}$ converges weakly to a positive multiple of
	the Dirac measure at $x$.
	Mazur's lemma implies that the weak limits of $\{-\eta_n^1\}$ and $\{\eta_n^2\}$
	belong to $\TT_M(q)$.
	Now, we can continue as in the proof of \cref{thm:dirichlet}.
	In particular, we can utilize the function $\varphi$, given by the assertion of the theorem,
	in the last step of the proof.
\end{proof}
Obviously,
many function spaces possess elements $y$ whose zero level sets are non-empty and (locally) of measure zero.
Similarly, functions $\varphi$ with the properties in \cref{thm:special_continuous} can easily be found in many situations. We do not go into details here but only give the following exemplary result.

\begin{corollary}
Let $\Omega \subset \R^d$ be a bounded domain with a Lipschitz boundary. Then, the  set $  \{ v \in L^\infty(\Omega) \mid  -1 \leq v \leq 1 \ \LL^d\text{-a.e.\ in } \Omega   \}  $ is non-polyhedric in the dual of the space $W^{s,p}(\Omega)$ for all $0 < s \leq 1$ and all $1 < p < \infty$ with $sp>d$.
\end{corollary}
\begin{proof}
From the Sobolev embeddings for fractional Sobolev spaces, cf.\ \cite[Theorem 8.2]{DINEZZA2012521}, it follows that $W ^{s,p}(\Omega)$ embeds continuously into the space $C(\overline{\Omega})$ and densely into $L^1(\Omega)$ for all $0 < s \leq 1$ and all $1 < p < \infty$ with $sp>d$. Further, functions $y$ and $\varphi$ with the properties in \cref{thm:special_continuous} are easily constructed (note that we may again set $q := \sign(y)$). Thus, the claim follows immediately.
\end{proof}

Finally, we briefly mention that the reflexivity assumption in \cref{thm:special_continuous} is crucial.
In particular, in the setting $\MM(X) := C_0(X)\dualspace$
the space $\MM(X)$ is a Banach lattice and we can apply
\cite[Theorem~4.18]{Wachsmuth2016:2}
to obtain the polyhedricity of $M \subset \MM(X)$,
see also \cite[Example~4.21(7)]{Wachsmuth2016:2}.

\section{Conclusion}
\label{sec:conclusion}

We conclude our analysis with two remarks on the results in \cref{subsec:DirichletSpaces,subsec:ContinuousSpaces}.

First, we would like to point out that the non-polyhedricity of the set $ \{ v \in L^\infty(\partial \Omega) \mid  -1 \leq v \leq 1 \ \SSS^{d-1}\text{-a.e.\ in } \partial \Omega   \}$ as a subset of $H^{-1/2}(\partial \Omega)$ in \cref{corollary:tangible1} implies that the assumptions on the regularity of the contact set typically made in the sensitivity analysis of frictional contact problems, cf.\ \cite[Lemma 4.24]{SokolowskiZolesio1992}, cannot be dropped. Additional assumptions on the solution or the involved sets are unavoidable when an approach analogous to that in \cref{sec:motivation} is used to study the directional differentiability of the solution operator to a $H^1$- or $H^{1/2}$-elliptic variational inequality of the second kind that involves a term of the form $\| v \|_{L^1(\partial \Omega)}$ (since, as we have seen, there are points where the condition of polyhedricity is violated). The same applies to $H^1$-elliptic variational inequalities involving a term of the form $\|v\|_{L^1(\Omega)}$ and the assumptions/the approach in \cite{DelosReyes2016}.

Second, it should be noted that the non-polyhedricity of the set \eqref{eq:BoundSetDual} in dual space cannot be overcome by resorting, e.g., to the concept of extended polyhedricity. As a matter of fact, the set  \eqref{eq:BoundSetDual} typically possesses positive curvature in the dual spaces.
In the space $H^{-1}(\Omega)$, this can be seen, e.g., in the results of \cite{ChristofMeyer2016}.

\ifbiber
	\printbibliography
\else
	\bibliographystyle{plainnat}
	\bibliography{references}
\fi
\end{document}